\documentclass[aop,MSNbibl,citesort,dvips]{arximspdf}

\doi{10.1214/12-AOP742} 
\volume{41}
\issue{3A}
\pubyear{2013}
\firstpage{1556}
\lastpage{1584}

\makeatletter

\newtheorem{prop}{Proposition}
\newtheorem{theo}[prop]{Theorem}

\newtheorem{theorem}{Theorem}
\newtheorem{corollary}{Corollary}[section]
\newtheorem{lemma}[corollary]{Lemma}
\newtheorem{proposition}[corollary]{Proposition}

\newcommand{\cal}{\mathcal}
\newcommand{\Prob}{{\mathbb P}}
\newcommand{\Z}{{\mathbb Z}}
\newcommand{\E}{{\mathbb E}}
\newcommand{\R}{{\mathbb{R}}}
\newcommand{\dist}{\operatorname{dist}}

\renewcommand{\Im}{\operatorname{Im}}
\renewcommand{\Re}{\operatorname{Re}}
\newcommand{\p}{\partial}

\newcommand{\Half}{{\mathbb H}}
\newcommand{\Disk}{{\mathbb D}}

\newcommand{\ball}{{\mathcal B}}

\newcommand{\F}{{\cal F}}
\newcommand{\domain}{{\cal D}}

\newproclaim{remark}{Remark}
\newproclaim{definition}{Definition}

\makeatother

\begin{document}
\begin{frontmatter}

\title{\textit{SLE} curves and natural parametrization}
\runtitle{Natural parametrization}

\begin{aug}
\author[A]{\fnms{Gregory F.} \snm{Lawler}\thanksref{t1}\ead[label=e1]{lawler@math.uchicago.edu}}
\and
\author[B]{\fnms{Wang} \snm{Zhou}\corref{}\thanksref{t2}\ead[label=e2]{stazw@nus.edu.sg}}
\runauthor{G. F. Lawler and W. Zhou}
\affiliation{University of Chicago and National University of Singapore}
\address[A]{Department of Mathematics\\
University of Chicago\\
5734 S. University Avenue\\
Chicago, Illinois 60637\\
USA\\
\printead{e1}} 
\address[B]{Department of Statistics\\
\quad and Applied Probability\\
National University of Singapore\\
Singapore 117546\\
Singapore\\
\printead{e2}}
\end{aug}

\thankstext{t1}{Supported by NSF Grant DMS-09-07143.}

\thankstext{t2}{Supported in part by Grant R-155-000-095-112 at
the National University of Singapore.}

\received{\smonth{2} \syear{2011}}
\revised{\smonth{12} \syear{2011}}

%
\begin{abstract}
Developing the theory of two-sided radial and chordal $\mathit{SLE}$,
we prove
that the natural parametrization on $\mathit{SLE}_\kappa$ curves is well
defined for all $\kappa< 8$. Our proof uses a two-interior-point
local martingale.
\end{abstract}

%

\end{frontmatter}

\section{Introduction}
\subsection{Background and motivation}

Suppose that $x_j, j=1,2,\ldots\,$, are independent and
identically distributed random vectors in $\Z^2$ with probabilities
\[
\Prob\{x_j=\mathrm{e}\}=1/4,\qquad |\mathrm{e}|=1.
\]
It is well known that the scaled simple random walk in $\R^2$,
\[
B_t^{(n)} = n^{-1/2}\sum_{j=1}^{[nt]}x_j,\qquad 0\leq t \leq1,
\]
converges to a standard
two-dimensional Brownian motion $B_t$, $0\leq t\leq1$, as \mbox{$n\to\infty$}.
In the scaled walk,
each step is traversed in the same amount of time.
When passing to the scaling limit, this parameter $t$ becomes
the natural parametrization of Brownian motion. We can write the
scaling factor as $n^{-1/d}$, where
$d = 2$ is the fractal dimension of the Brownian paths. The natural
parametrization
is a $d$-dimensional measure.

One variant of simple random walk is the loop erased random walk first
appeared in~\cite{lawlerwalk}. Its definition is as follows. Consider
any finite or recurrent connected graph $G$, one vertex $a$ and a set of
vertices $V$. Loop-erased random walk (LERW) from $a$ to $V$ is a
random simple
curve joining $a$ to $V$ obtained by erasing the loops in chronological
order from
a simple random walk started at a and stopped upon hitting $V$.
One can ask whether or not there is a corresponding result for LERW
where the scaling
factor is $n^{-1/d}$ and $d$ is the fractal dimension of the paths.
Schramm~\cite{Schramm} introduced a process, now called the
Schramm--Loewner evolution
($\mathit{SLE}_\kappa$),
as a candidate for the scaling limit and gave a strong argument why $\mathit{SLE}_2$
should be the scaling limit of LERW. In order to use the Loewner
equation, he
used a capacity parametrization which is not the
parametrization one would obtain by taking the limit above.
In~\cite{LSWlerw}, it was proved that the scaling limit of planar LERW
\textit{in the capacity parametrization}
is $\mathit{SLE}_2$. However, it is still open whether or not, one can take a limit
as above. There are a number of other models that are known to converge
to $\mathit{SLE}_\kappa$ in the scaling limit using the capacity parametrization:
critical site-percolation on the triangular lattice~\cite{Sperco}, the
level lines of the discrete Gaussian free field~\cite{SSfree},
the interfaces of the random cluster model associated with the Ising
model~\cite{Sising}.

A start to taking limits as above is to define the natural parametrization
for $\mathit{SLE}$.
Possible definitions and constructions for a natural parametrization were
suggested in~\cite{LStime}. As well as giving conjectures, one definition
was proposed in terms of the Doob--Meyer decomposition of a path. We review
this construction in Section~\ref{review}. Although they conjectured
that this
definition is valid for all $\kappa< 8$, they were only able to
establish the
result for $\kappa< \kappa_0 = 4(7-\sqrt{33})$. The technical problem
came from difficult second moment estimates for the reverse Loewner flow.

In this paper, we prove that the definition in~\cite{LStime}
is valid for $\kappa< 8$.
Instead of using the reverse Loewner flow, we use a difficult estimate of
Beffara~\cite{Beffara} on the forward Loewner flow to establish the necessary
uniform integrability to apply the Doob--Meyer theorem. Beffara's estimate
was the key step in his proof of the Hausdorff dimension of
$\mathit{SLE}_\kappa$ curves. This estimate has recently been improved~\cite{LWerness}
and used to
establish a multi-point Green's function for $\mathit{SLE}_\kappa$. We use this Green's
function to give an appropriate two-interior-point local martingale. By
establishing
a correlation inequality for this Green's function, we are able to give a
relatively simple proof of the existence of the natural parametrization.

\subsection{Notation}

In this subsection, we set up the notation for $\mathit{SLE}_\kappa$. For more
background, see, for example, \cite
{wernernotes,GKnotes,Cardynotes,lawlerbook,LPark}.

Throughout this paper, we let
$\kappa< 8$ and $a = 2/\kappa> 1/4$, and we allow all
constants to depend on $\kappa$. We let $d =
1 + \frac\kappa8 = 1 + \frac1{4a}$ be the Hausdorff
dimension of the paths. We parametrize
the maps so that
%
\begin{equation} \label{loeweq}
\p_tg_t(z) = \frac{a}{g_t(z) -U_t},
\end{equation}
where $U_t = - B_t$ is a standard Brownian motion. It can be
shown~\cite{RS,LSWlerw} that a.s. $g_t^{-1}$ extends continuously
to $\overline\Half$ for every $t\ge0$ and $\gamma
(t):=g_t^{-1}(U_t)$ is a continuous curve
which is the $\mathit{SLE}$ path. The domain of definition of $g_t$ is the unbounded
connected component $H_t$ of $\Half\setminus\gamma[0,t]$.
We shall denote by $K_t$ the closure of the complement of $H_t$ in
$\Half$. If $z \in\overline
\Half\setminus\{0\}$, let
\[
Z_t(z) = X_t(z) + i Y_t(z) =
g_t(z) + B_t.
\]
Then the Loewner equation can be written as
\begin{eqnarray*}
dZ_t(z) &=& \frac{a}{Z_t(z)} \, dt + dB_t,
\\
dX_t(z) &=& \frac{a X_t(z)}{|Z_t(z)|^2} \, dt + dB_t,\qquad
\p_t Y_t(z) = - \frac{a Y_t}{|Z_t|^2}.
\end{eqnarray*}
The Loewner equation is valid up to the time
\[
T_z = \sup\{t\dvtx Y_t(z) > 0\}.
\]
Let
\[
\Upsilon_t(z) = \frac{Y_t(z)}{|g_t'(z)|},\qquad
\theta_t(z) = \arg Z_t(z),\qquad
S_t(z) = \sin\theta_t(z) = \frac{Y_t(z)}{|Z_t(z)|}.
\]
It is not difficult to see that
$\Upsilon_t(z)$ is $1/2$ times the conformal radius of $H_t$
with respect to $z$, by which we mean
that if $f\dvtx\Disk\rightarrow H_t$ is a conformal transformation with
$f(0) = z$, then $|f'(0)| = 2 \Upsilon_t(z)$. Using
the Schwarz lemma and
the Koebe $(1/4)$-theorem, we can see that
\[
\Upsilon_t(z) \asymp_2 \dist\bigl(z,\R\cup\gamma(0,t]\bigr),
\]
where $\asymp_2$ means that both sides are bounded above by
$2$ times the other side.
Using the Loewner equation (\ref{loeweq}), we see that
%
\begin{equation} \label{may224}
\p_t \Upsilon_t(z) = - \Upsilon_t(z) \frac{2a Y_t(z)^2}
{|Z_t(z)|^4}.
\end{equation}
In particular, $\Upsilon_t(z)$ decreases
with $t$ and hence we can define
\[
\Upsilon(z) = \lim_{t \rightarrow T_z-}
\Upsilon_t(z),
\]
which satisfies
\[
\Upsilon(z) \asymp_2 \dist[z,\R\cup\gamma(0,\infty)].
\]
Using It\^o's formula, we can see that
%
\begin{equation} \label{may225}
d \theta_t(z) = \frac{(1-2a) X_t(z) Y_t(z)}{
|Z_t(z)|^4} \, dt - \frac{Y_t(z)}{|Z_t(z)|^2}
\, dB_t.
\end{equation}
The \textit{Green's function} (for $\mathit{SLE}_\kappa$ from $0$
to $\infty$ in $\Half$) is defined by
\[
G(z) = y^{d-2} [\sin\arg z ]^{4a-1}
= y^{{1}/({4a}) + 4a-2} |z|^{1-4a},
\]
where $z = x+iy = |z| e^{i\theta}$.
It\^o's formula shows that
%
\begin{equation} \label{jun81}
M_t(z) = |g_t'(z)|^{2-d} G(Z_t(z))
= \Upsilon_t(z)^{d-2} S_t(z)^{4a-1},\qquad
t < T_z,
\end{equation}
is a local martingale satisfying
\[
dM_t(z) = \frac{ (1-4a) X_t(z)}{|Z_t(z)|^2} M_t(z) \,
dB_t.
\]

More generally, if $D$ is a simply connected domain, $z \in D$,
and $w_1,w_2$ are distinct points in $\p D$, we can define
$\Upsilon_D(z), S_D(z;w_1,w_2)$ using the following
scaling rules: if $f\dvtx D \rightarrow f(D)$ is
a conformal transformation, then
\[
\Upsilon_{f(D)}(f(z)) = |f'(z)| \Upsilon_D(z),\qquad
S_{f(D)}(f(z);f(w_1),f(w_2)) =
S_D(z;w_1,w_2).
\]
The Green's function $G_D(z;w_1,w_2)$ is defined by
\[
G_D(z;w_1,w_2) = \Upsilon_D(z)^{d-2} S_D(z;w_1,w_2)
^{4a-1}
\]
and satisfies the
scaling rule
%
\begin{equation} \label{greenscale}
G_{D}(z;w_1,w_2) = |f'(z)|^{2-d}
G_{f(D)}(f(z);f(w_1),f(w_2)).
\end{equation}
Under this definition, the local martingale in (\ref{jun81})
can be rewritten as
\[
M_t(z) = G_{H_t}(z;\gamma(t),\infty).
\]
The following easy lemma is useful for estimating
$S_D(z;w_1,w_2)$.
%
\begin{lemma} Suppose $D$ is a simply connected domain,
$w_1,w_2 \in\p D$, and $f\dvtx\Half\rightarrow D$ is a
conformal transformation with $f(0) = w_1, f(\infty)
= w_2$. Let $\p_+ = f[(0,\infty)], \p_- = f[(-\infty,0)]$.
If $z \in D$, let
\[
q = q_D(z;w_1,w_2) =
\min\{h_D(z,\p_+), h_D(z,\p_-)\},
\]
where $h_D$ denotes harmonic measure. Then
%
\begin{equation} \label{jun91}
2q \leq S_D(z;w_1,w_2) \leq\pi q.
\end{equation}
\end{lemma}
\begin{pf} We first
note that $q,h_D,S_D$ are conformal invariants, so it suffices
to prove the result for $D = \Half, w_1 = 0, w_2 = \infty$ and
by symmetry we may assume that $\theta:= \arg z \leq\pi/2$.
By explicit calculation, we can see that $\theta= \pi q$, and hence
(\ref{jun91}) follows from the estimate
\[
\frac{2}{\pi} \leq\frac{\sin x}x \leq
1,\qquad 0 <x \leq\frac\pi2.
\]
\upqed\end{pf}

Using Girsanov's theorem, it can be shown that
\[
\lim_{\epsilon\rightarrow0+} \epsilon^{d-2} \Prob\{\Upsilon_\infty(z)
\leq\epsilon\} = c_* G(z),\qquad c_* = 2 \biggl[{\int_0^\pi\sin^{4a} x \,
dx}\biggr]^{-1}.
\]
A proof of this is given in~\cite{LPark}, but we include
a self-contained proof in this paper (Proposition
\ref{junprop1}) which also estimates the error term.
It follows that if $D$ is a simply connected domain, $z \in D$;
$\gamma$ is an $\mathit{SLE}_\kappa$ curve connecting distinct boundary
points $w_1,w_2 \in\p D$, and $D_\infty$ denotes the component of
$D \setminus\gamma$ containing $z$, then
\[
\Prob\{\Upsilon_{D_\infty}(z) \leq\epsilon\}
\sim c_* \epsilon^{2-d}
G_D(z;w_1,w_2),\qquad \epsilon\rightarrow0{+}.
\]

\subsection{Review of natural parametrization} \label{review}
We will briefly review the construction in~\cite{LStime}. The starting
point is the following proposition.
\renewcommand{\theprop}{\Alph{prop}}
\begin{prop}[(\cite{LStime})]\label{propA}
Suppose that there exists a parametrization for
$\mathit{SLE}_\kappa$ in $\Half$ satisfying the domain Markov
property and the conformal invariance assumption. For a fixed Lebesgue
measurable subset $S \subset\Half$,
let $\Theta_t(S)$ denote the process that gives the amount of time in
this parametrization
spent in $S$ before time $t$
(in the half-plane capacity parametrization), and suppose further that
$\Theta_t(S)$ is
$\mathcal F_t$ adapted for all such $S$. If $\mathbb E \Theta_\infty
(D)$ is finite for
all bounded domains $D$, then it must be the case that (up to
multiplicative constant)
\[
\E\Theta_\infty(D) = \int_D G(z) \, dA(z),
\]
where $dA$ denotes integration with respect to area, and more
generally,
\[
\mathbb E[\Theta_\infty(D) - \Theta_t(D)|
\mathcal F_t] = \int_D M_t(z)\, dA(z).
\]
\end{prop}

Let $\domain$ denote the set of bounded domains $D \subset
\Half$ with $\dist(\R,D) > 0$. Write
\[
\domain= \bigcup_{m=1}^\infty
\domain_m,
\]
where $\domain_m$ denotes the set of domains $D$ with
\[
D \subset\{x+iy\dvtx |x| < m, 1/m < y < m\}.
\]
For any process $\Theta_t(D)$ with finite expectations, by Proposition
\ref{propA}, one has
%
\begin{equation} \label{Psi}
\Psi_t(D) =
\E[\Theta_\infty(D) |\F_t] - \Theta_t(D),
\end{equation}
where $\Psi_t(D):=\int_D M_t(z)\, dA(z)$,
which is a supermartingale in $t$ because $M_t(z)$ is a nonnegative
local martingale. It
is also not difficult to prove that $\Psi_t(D)$ is in fact continuous
as a function of $t$ by its definition.
Assuming the conclusion of Proposition~\ref{propA}, the first term on the
right-hand side of (\ref{Psi}) is a martingale and the map
$t \mapsto\Theta_t(D)$ is increasing.
Inspired by the continuous case of the standard\vadjust{\goodbreak}
Doob--Meyer theorem~\cite{DM}: any
continuous supermartingale can be written uniquely as the sum of a
continuous adapted decreasing process
with initial value zero and a continuous local martingale, Lawler and
Shieffeld in~\cite{LStime} introduce the following definition.
\begin{definition*}[(\cite{LStime})]
\begin{itemize}
\item
If $D \in\domain$, then the natural parametrization
$\Theta_t(D)$ is the unique continuous, increasing process such
that
\[
\Psi_t(D) + \Theta_t(D)
\]
is a martingale (assuming such a process exists).

\item If $\Theta_t(D)$ exists for each $D \in\domain$,
the natural parametrization in $\Half$ is given by
\[
\Theta_t = \lim_{m \rightarrow\infty}
\Theta_t(D_m),
\]
where
$D_m = \{x+iy\dvtx |x|<m, 1/m < y < m\}. $
\end{itemize}
\end{definition*}

The main result in that paper is the following.
\begin{theo}[(\cite{LStime})]\label{theoB}
If $\kappa< \kappa_0:= 4(7-\sqrt{33})$,
there is an adapted,
increasing, continuous process $\Theta_{t}(D)$ with $\Theta_0(D) = 0$
such that
\[
\Psi_{t}(D) + \Theta_{t}(D)
\]
is a martingale. Moreover, with probability one for all $t$
%
\begin{eqnarray} \label{aug141}
\Theta_{t}(D) &=& \lim_{n \rightarrow\infty}
\sum_{j \leq t2^n}
\int_\Half\bigl|
\hat f_{({j-1})/{2^{n}}}'(z)\bigr|^d \phi(z 2^{n/2})
G(z)\nonumber\\[-8pt]\\[-8pt]
&&\hspace*{59pt}{}\times
1\bigl\{\hat f_{({j-1})/{2^{n}}}(z) \in D\bigr\}\, dA(z),\nonumber
\end{eqnarray}
where
$\phi(z)$ is defined by $\E[M_1(z)]=M_0(z)(1-\phi(z))$ and \mbox{$\hat
f_s(z)=g_s^{-1}(z+U_s)$}.
\end{theo}

As for the proof, they start by discretizing time and finding an
approximation for
$\Theta_t(D)$. This time discretization is the first step in proving
the Doob--Meyer decomposition
for any supermartingale. The second step is to take the limit. For this purpose,
they use the reverse-time flow for the Loewner equation to derive
uniform second moment estimates for the approximations when $\kappa
<\kappa_0$.
This estimate is the most difficult one in all their derivation.
Then they can take a limit both in $L^2$ and with probability one.

\subsection{Multi-point Green's function}

As the main step in proving the Hausdorff dimension of the
$\mathit{SLE}$ curve, Beffara~\cite{Beffara} proved the following lemma.

\begin{lemma} Suppose $D$ is a bounded subdomain
of $\Half$ with \mbox{$\dist(D,\R) > 0$}. Then there exists $c_D <\infty$
such that if $z,w \in D$ and $\epsilon> 0$,
\[
\Prob\{\Upsilon_\infty(z) \leq\epsilon, \Upsilon_\infty
(w) \leq\epsilon\} \leq c_D \epsilon^{2(2-d)} |z-w|^{d-2}.\vadjust{\goodbreak}
\]
\end{lemma}

Recently, Lawler and Werness~\cite{LWerness} extended Beffara's
argument to show that
%
\begin{equation} \label{befgen}
\Prob\{\Upsilon_\infty(z) \leq\epsilon, \Upsilon_\infty
(w) \leq\delta\} \leq c_D \epsilon^{2-d} \delta^{2-d}
|z-w|^{d-2}.
\end{equation}
Building on this, they show that there is a multi-point Green's
function $G(z,w)$ such that
%
\begin{equation} \label{multipoint}
\lim_{\epsilon,\delta\rightarrow0+} \epsilon^{d-2}
\delta^{d-2}
\Prob\{\Upsilon_\infty(z) \leq\epsilon,
\Upsilon_\infty(w)\leq\delta\} = c_*^2 G(z,w).
\end{equation}
Although a closed form of the function $G(z,w)$ is not given,
it is shown that
%
\begin{equation} \label{may271}
G(z,w) = G(z) G(w) [F(z,w) + F(w,z)],
\end{equation}
where
\[
F(z,w) = \frac{ \E_z^* [|g_T'(w)|^{2-d}
G(Z_T(w)) ]}{G(w)},
\]
$\E^*_z$ denotes expectation with respect to
two-sided radial $\mathit{SLE}_\kappa$ through $z$ (see~Section \ref
{tworadial} for
definitions) and $T = T_z = \inf\{t\dvtx \gamma(t) = z\}$.
Roughly speaking, $G(z,w)$ represents the probability of
going through $z$ and $w$, and\break $G(z) G(w) F(z,w)$ represents
the probability of going first through $z$ and then
through $w$. Using (\ref{greenscale}), we can write
\[
G(w) F(z,w) = \E_z^* [G_{D_T}(w;z,\infty)
].
\]
From the definition, we can see that if $r > 0$,
\[
G(z,w) = r^{2(2-d)} G(rz,rw),\qquad
F(z,w) = F(rz,rw).
\]

More generally, if $D$ is a simply connected domain with
boundary points $z_1,z_2$, we can define $F_D(z,w;
z_1,z_2)$ by conformal invariance.

Let
%
\begin{eqnarray} \label{may301}
M_t(z,w) &=& |g_t'(z)|^{2-d} |g_t'(w)|^{2-d}
G(Z_t(z),Z_t(w))\nonumber\\[-8pt]\\[-8pt]
&=& \frac{M_t(z) M_t(w) G(Z_t(z),Z_t(w))
}{G(Z_t(z)) G(Z_t(w))}.\nonumber
\end{eqnarray}
This is the so-called two-interior-point local martingale. A
similar two-boundary-point local martingale appears in \cite
{schrammzhoudimension}.
Using (\ref{multipoint}), we can see if $\epsilon> 0$
and $T_\epsilon$ is the first time such that $\Upsilon_{T_\epsilon}(z)
\leq\epsilon$ or $\Upsilon_{T_\epsilon}(w)
\leq\epsilon$, then $\E[M_{T_\epsilon}(z,w)] = M_0(z,w) =
G(z,w)$. Hence, by Fatou's lemma, for every stopping
time $T$,
%
\begin{equation} \label{may272}
G(z,w) \geq\E[M_T(z,w)].
\end{equation}

For our
main result, we will need the following two estimates about
$G(z,w)$. The first follows immediately from (\ref{befgen})
and (\ref{multipoint}); establishing the second is the main
technical work in this paper.
%
\begin{lemma}$ $ \label{mainlemma}
\begin{itemize}
\item
Suppose $D$ is a bounded subdomain
of $\Half$ with $\dist(D,\R) > 0$. Then there exists $c_D <\infty$
such that if $z,w \in D$,
\[
G(z,w) \leq c_D |z-w|^{d-2}.
\]

\item There exists $c > 0$ such that for all $z,w \in\Half$,
%
\begin{equation} \label{may274}
G(z,w) \geq c G(z) G(w).
\end{equation}
\end{itemize}
\end{lemma}

Note that (\ref{may274}) is equivalent to saying that there
exists $c$ such that
\[
F(z,w) + F(w,z) \geq c.
\]
We remark that
\[
\inf_{z,w} F(z,w) = 0.
\]
Indeed, one can check that
\[
\lim_{y \rightarrow\infty} F(iy,i/y) =0.
\]
The
basic idea is that if $y$ is large then the chance that the
$\mathit{SLE}$ path goes through $yi$ \textit{and then through} $i/y$ is much
smaller than the probability of going through $i/y$ and then through
$iy$.
%
\begin{corollary} If $D$ is a bounded subdomain
of $\Half$ with $\dist(D,\R) > 0$, then there exists $c_D <\infty$
such that if $z,w \in D$, and $T$ is a stopping time,
%
\begin{equation} \label{may273}
\E[M_T(z) M_T(w)] \leq c_D |z-w|^{d-2}.
\end{equation}
\end{corollary}
\begin{pf} Using (\ref{may301}), (\ref{may272}), and
Lemma~\ref{mainlemma}, we get
\[
\E[M_T(z) M_T(w)] \leq
c \E[M_T(z,w)] \leq c G(z,w) \leq
c_D |z-w|^{d-2}.
\]
\upqed\end{pf}

\subsection{The main theorems}

As in Lawler and Sheffield,
we will prove that there is an adapted,
increasing, continuous process $\Theta_{t}(D)$ with $\Theta_0(D) = 0$
such that $\Psi_{t}(D) + \Theta_{t}(D) $ is a martingale. The basic
idea of the proof is the same.
Here, we show how (\ref{may273})
yields
the
uniform integrability (class $\mathfrak{D}$) needed to establish
the existence of the martingale.
%
\begin{theorem} \label{main}
If $0<\kappa< 8$,
there is an adapted,
increasing, continuous process $\Theta_{t}(D)$ with $\Theta_0(D) = 0$
such that
\[
\Psi_{t}(D) + \Theta_{t}(D)
\]
is a martingale. Moreover, let
\[
\Theta_{t,n}(D) =
\sum_{j \leq t2^n}
\int_\Half\bigl|
\hat f_{({j-1})/{2^{n}}}'(z)\bigr|^d \phi(z 2^{n/2}) G(z)
1\bigl\{\hat f_{({j-1})/{2^{n}}}(z) \in D\bigr\} \, dA(z),\vadjust{\goodbreak}
\]
then for any stopping time $T$,
\[
\lim_{n \to\infty} \E[|\Theta_{T,n}(D)- \Theta_T(D)
| ]=0.
\]
\end{theorem}
\begin{remark*}
From Theorem~\ref{theoB}, we know that $\E M_1(z)=M_0(z)(1-\phi(z))< M_0(z)$. So
$M_t(z)$ is a
local martingale and a supermartingale
which is not a proper martingale. This implies that $ \Psi_{t}(D) $ is
not a proper martingale too.
Since the theorem establishes that
$\Psi_{t}(D) + \Theta_{t}(D) $ is actually a martingale,
$\Theta_{t}(D)$ is nontrivial. In other words, it is not identically zero.
\end{remark*}

Before discussing how to prove Theorem~\ref{main}, let us briefly
review Doob--Meyer decomposition for supermartingales of class
$\mathfrak{D}$. First, we assume that $\Psi_t$ is a supermartingale
with respect to a filtration ${\cal F}_t$, defined on the interval $[0,
\infty)$. We also suppose that ${\cal F}_t$ satisfies the usual
conditions. The supermartingale $\{\Psi_t, {\cal F}_t, t\geq0\}$ is
said to be of \textit{class $\mathfrak{D}$} if the family $ \{ \Psi
_{T}\dvtx T$ is an almost surely finite stopping time$\}$ is
uniformly integrable. The following result is about Doob--Meyer
decomposition for supermartingales of class~$\mathfrak{D}$. One can
refer to Section 1.4 of~\cite{karatzasshreve} for more details.
\begin{theo}[(Doob--Meyer decomposition~\cite{karatzasshreve})]\label{theoC}
Let $\{\Psi_t, {\cal F}_t, t\geq0\}$ be a continuous
supermartingale of class $\mathfrak{D}$.
Then there exists a continuous predictable, nondecreasing process $\{
A_t, {\cal F}_t, t\geq0\}$, such that
$A_0 = 0$, $A_{\infty}$ is integrable, and
\[
\Psi_t = \mathcal{M}_t - A_t,
\]
where $\mathcal{M}_t:=\E[A_{\infty} + \Psi_{\infty}|{\cal F}_t]$
is a
martingale. This decomposition is unique
up to indistinguishability, that is, if $\{\mathcal{M}_t',{\cal F}_t,
t\geq0\}$ and
$\{A_t', {\cal F}_t, t\geq0\}$ are a martingale and a predictable,
nondecreasing process
satisfying the above
properties, respectively, then
\[
\Prob\{\mathcal{M}_t = \mathcal{M}_t', A_t = A_t', \forall
t\geq0\} = 1.
\]
\end{theo}

In order to prove Theorem~\ref{main}, we need one result from Section
20 of Chapter VII in~\cite{DM} about Doob--Meyer decompositions of an
increasing sequence of positive supermartingales.
\begin{theo}[(\cite{DM})]\label{theoD}
Let $\{\Psi_t^n, {\cal F}_t, t\geq
0\}$ be a sequence of positive supermartingales, whose limit
$\{\Psi_t, {\cal F}_t, t\geq0\}$ belongs to class $\mathfrak{D}$ and
is regular. Let $A_t^n$ and $A$ denote the nondecreasing processes
associated with $\Psi_t^n$ and $\Psi_t$, respectively. Then for any
stopping time $T$,
\[
\lim_{n \to\infty} \E[| A_T^n-A_T|]=0.
\]
\end{theo}
\begin{pf*}{Proof of Theorem~\ref{main} given Lemma~\ref{mainlemma}}
From Theorem~\ref{theoC}, in order to show the existence of $\Theta_{t}(D)$ with the
desired properties,\vadjust{\goodbreak} we need to find a $c(D) < \infty$ such that for
every stopping time $T$,
%
\begin{equation} \label{u-b}
\E[\Psi_T^2(D)] \leq c(D).
\end{equation}
From (\ref{may273}), we have
\begin{eqnarray*}
\E[\Psi_T^2(D)] & = & \int_D\int_D \E[M_T(z_1)M_T(z_2)] \,dA(z_1)\,dA(z_2)
\\& \leq& c_D
\int_D\int_D |z_1-z_2|^{d-2} \,dA(z_1)\,dA(z_2) \leq c(D).
\end{eqnarray*}
This gives the first result. Now we turn to the second result. For all
$t$, we set
\[
\Psi_t^n(D)=\E\bigl[\Psi_{(i+1)2^{-n}}(D)|{\cal F}_t\bigr] \qquad\mbox{if }
i2^{-n} \leq t< (i+1)2^{-n}.
\]
Hence, $\{\Psi_t^n(D), {\cal F}_t, t\geq0\}$ is a positive
supermartingale bounded above by $\Psi_t(D)$, which increases to $\Psi
_t(D)$ as $n\to\infty$. Let $A_t^n$ be the nondecreasing process
associated with $\Psi_t^n(D)$. Then for $i2^{-n} \leq t< (i+1)2^{-n}$,
\[
A_t^n=\sum_{0\leq k<i}\E\bigl[\Psi_{k2^{-n}}(D)-\Psi
_{(k+1)2^{-n}}(D)|{\cal F}_{k2^{-n}}\bigr].
\]
Actually, by the change of variables and the scaling rule of $\phi$,
we have
\[
A_t^n=\Theta_{t,n}(D).
\]
By Theorem~\ref{theoD}, we can complete the proof.
\end{pf*}
\begin{remark*} Although we have used the existence of the multi-point
Green's function $G(z,w)$ as established in~\cite{LWerness}, we could have
proven our result here without its existence. In fact, an earlier
draft of our paper derived the theorem from Beffara's estimate
replacing
(\ref{may274}) with
\[
\Prob\{\Upsilon(z) \leq\epsilon, \Upsilon(w) \leq\delta
\}
\geq c \Prob\{\Upsilon(z) \leq\epsilon\}
\Prob\{\Upsilon(w) \leq\delta\}.
\]
However, the argument is cleaner when written in terms of $G(z,w)$
so we use it here.
\end{remark*}

The next theorem shows that our natural parametrization can be seen as
conformal Minkowski measure (although is not
the same definition as the conformal Minkowski content defined
in~\cite{LStime}).
%
\begin{theorem} \label{main2}
Let $\gamma^{\epsilon}(0,t]=\{z\in\Half\dvtx\sup_{0\leq s\leq
t}M_s(z)\geq\epsilon^{d-2}\}$.
If $0<\kappa< 8$, then for any stopping time $T$,
\[
\lim_{\epsilon\to0} \E\bigl[\bigl|\epsilon^{d-2}A\bigl(D\cap
\gamma^{\epsilon}(0,T]\bigr)-\Theta_T(D) \bigr| \bigr]=0.
\]
\end{theorem}
\begin{pf*}{Proof of Theorem~\ref{main2} given Lemma~\ref{mainlemma}}
Note that
\[
\epsilon^{d-2}A\bigl(D\cap\gamma^{\epsilon}(0,t]\bigr)=
\int_D M_{t\wedge\tilde\tau_{\epsilon}(z)}(z)I\{z\in\gamma
^{\epsilon}(0,t] \}\,dA(z),\vadjust{\goodbreak}
\]
where $\tilde\tau_{\epsilon}(z)=\inf\{t\geq0\dvtx M_t(z)\geq
\epsilon^{d-2}\}$. The above integral appears in the identity
\begin{eqnarray*}
&&\int_D M_{t\wedge\tilde\tau_{\epsilon}(z)}(z)I\{z\notin
\gamma^{\epsilon}(0,t] \}\,dA(z) \\
&&\qquad=\int_D M_{t\wedge\tilde\tau_{\epsilon}(z)}(z)\,dA(z)-\int
_D M_{t\wedge\tilde\tau_{\epsilon}(z)}(z)I\{z\in\gamma
^{\epsilon}(0,t] \}\,dA(z).
\end{eqnarray*}
This is actually the Doob--Meyer decomposition of the
supermatingale\break
$\int_D M_{t\wedge\tilde\tau_{\epsilon}(z)}(z)I\{z\notin
\gamma^{\epsilon}(0,t] \}\,dA(z)$. By Theorem~\ref{theoD} and (\ref{u-b}), we
can complete the proof.
\end{pf*}

\subsection{Outline of the paper}

In Section~\ref{tworadial}, we will develop the theory of two-sided
radial $\mathit{SLE}$ in order to prove
Proposition~\ref{hullprop}. Although this was discussed somewhat in
\cite{LPark},
our treatment here is self-contained. One main goal is the refined ``one-point''
estimate in Proposition~\ref{junprop1}. The first step of the proof
goes back
to~\cite{RS} in which the radial parametrization
was used to establish a weaker form of the estimate. Here we use
Girsanov's theorem to
reduce the question to the rate of convergence to equilibrium of a simple
one-dimensional diffusion. The sharp estimate is used
in~\cite{LWerness} to improve Beffara's estimate and show existence of the
multi-point Green's function.
When considering radial $\mathit{SLE}_\kappa$
going through points $z = x+iy$ with $|x| \gg y$, it is useful
to compare this to a process headed to $x$. We call the
appropriate one-dimensional process
``two-sided chordal $\mathit{SLE}_\kappa$'' and give some of its
properties.

The only thing we need to prove
Theorem~\ref{main} is the estimate (\ref{may274}) which will be
handled in Section~\ref{Section3}.

\section{Two-sided radial $\mathit{SLE}$} \label{tworadial}

One can see that with probability one\break $M_{T_z-}(z) = 0$; indeed, with
probability one $\Upsilon(z) >0$ and $S_{T_z-}(z) = 0$.
This local
martingale blows up on the event of probability zero that
$z \in\gamma(0,\infty)$. To be more precise, let
\[
\tau_\epsilon(z) = \inf\{t\dvtx \Upsilon_t(z) \leq
\epsilon\}.
\]
Then
$M_{t \wedge\tau_\epsilon(z)}(z)$ is a martingale.
For convenience, we will fix $z$ and write $M_t,X_t,Y_t,\tau_\epsilon,
\ldots$ for
$M_t(z), X_t(z),Y_t(z),\tau_\epsilon(z),\ldots\,$.

\textit{Two-sided radial $\mathit{SLE}_\kappa$
through $z$} is the process
obtained from the Girsanov transformation by weighting by the
local martingale $M_t $. We should really call this two-sided
radial $\mathit{SLE}_\kappa$ from $0$ to $\infty$ going through $z$ stopped
when it reaches $z$, but we will just say two-sided
radial $\mathit{SLE}_\kappa$ through $z$.
This could also be called $\mathit{SLE}_\kappa$
conditioned to go through $z$ (and stopped at $T_z$)
although this is a conditioning on
an event of probability zero. Using (\ref{jun81})
and the Girsanov theorem, we see that for each $\epsilon$
%
\begin{equation} \label{may223}
dB_t = \frac{ (1-4a) X_t }{|Z_t |^2} \, dt
+ dW_t,\qquad t \leq\tau_\epsilon,
\end{equation}
where $W_t$ is a standard Brownian motion in the weighted measure
$\Prob^*$ which can be defined by saying that if $V$ is an event
depending only on $\{B_t\dvtx\break 0 \leq t \leq\tau_\epsilon\}$,\vspace*{-2pt}
\[
\Prob^*(V) = \Prob^*_z(V) = G(z)^{-1}
\E[M_{\tau_\epsilon} 1_V].
\]
We write $\E^*$ for expectations with respect to $\Prob^*$.
There is an implicit $z$ dependence
in $\Prob^*$ and $\E^*$; when we need to make this explicit,
we write $\Prob^*_z,\E^*_z$. It is not hard to show (see Section
\ref{confsec}) that for every
$\epsilon> 0$, $\Prob^*\{\tau_\epsilon< \infty\} = 1$. Since
the SDE in (\ref{may223}) has no $\epsilon$ dependence, we
can let $0 \leq t < T_z$.
Note that\looseness=-1
\begin{eqnarray*}
dX_t &=& \frac{ (1-3a) X_t }{|Z_t |^2} \, dt
+ dW_t,
\\[-3pt]
-U_t &=& B_t = X_t - a \int_0^t
\frac{X_s }{|Z_s |^2} \, ds.\vspace*{-3pt}
\end{eqnarray*}\looseness=0

\subsection{Radial parametrization} \label{confsec}

When studying the behavior of an $\mathit{SLE}_\kappa$ curve near an
interior point $z$, it is useful to reparametrize the curve so that
$\log\Upsilon_t(z)$ decays linearly. This is the parametrization
generally used for radial $\mathit{SLE}$ and for this reason we call it
the \textit{radial parametrization}.
We will study the radial
parametrization in this subsection. We denote this
time change
as a function $\sigma(t)$ and we write
$\hat X_t = X_{\sigma(t)},
\hat Y_t = Y_{\sigma(t)}$, etc.

We define $\sigma$ by asserting that
\[
\hat\Upsilon_t = \Upsilon_{\sigma(t)} = e^{-2at}.\vspace*{-2pt}
\]
Using (\ref{may224}), we see that
\[
-2a \hat\Upsilon_{t} = \p_t[\hat\Upsilon_t
] = -\hat\Upsilon_t \frac{2a \hat Y_t^2}{|\hat
Z_t|^4} [\p_t \sigma(t)]
\]
and hence
\[
\p_t \sigma(t) = \frac
{|\hat Z_t |^4}{\hat Y_t ^2}.
\]
From (\ref{may225}), we see that
$\hat\theta_t:=
\theta_{\sigma(t)}$ satisfies
%
\begin{equation} \label{may226}
d\hat\theta_t = 2a \cot\hat\theta_t \, dt
+ d \hat W_t,
\end{equation}
where $\hat W_t$ is a standard Brownian motion with respect
to the weighted
measure $\Prob^*$. Since $a > 1/4$, the process (\ref{may226})
stays in $(0,\pi)$ for all times. This implies that for
all $\epsilon> 0$,
%
\begin{equation} \label{may227}
\Prob^*\{\tau_\epsilon< \infty\} = 1.
\end{equation}

\subsubsection{The corresponding SDE}

The analysis of two-sided radial $\mathit{SLE}$ relies on detailed
properties of (\ref{may226}) which fortunately
are not very difficult to obtain. In this
subsection, we focus on this SDE and set
$r =2a> 1/2$. Constants in this section
may depend
on $r$. We consider the equation
%
\begin{equation} \label{may222alt}
dX_t = r \cot X_t \,dt + dB_t,\qquad
X_0 \in(0,\pi).\vadjust{\goodbreak}
\end{equation}
In studying this equation, it is useful to note that this
is the equation that one gets if one starts with a Browian motion
$X_t$ and then weights paths locally by $\sin^r X_t$. To
be more precise,
if $X_t$ is a standard Brownian motion,
then
\[
M_t = [\sin X_t]^r \exp\biggl\{-\frac{r(r-1)}{2}
\int_0^t \frac{ds}{ \sin^2 X_s} +\frac12 r^2 t\biggr\}
\]
is a local martingale satisfying
\[
dM_t = r [\cot X_t] M_t \, dX_t.
\]
If we weight by the local martingale, then $X_t$ satisfies
(\ref{may222alt}) where $B_t$ is a Brownian motion in the
weighted measure. In fact, since the weighted process
stays in $(0,\pi)$, one can see that $M_t$ is actually
a martingale.

The equation (\ref{may222alt})
is closely related to the Bessel
equation
%
\begin{equation} \label{bessel}
dX_t = \frac{r}{X_t} \, dt + dB_t.
\end{equation}
This equation is obtained by starting with a Brownian
motion $X_t$ and weighting locally by $X_t^r$, that is,
by the martingale
\[
N_t = X_t^r \exp\biggl\{- \frac{r(r-1)}2
\int_0^t
\frac{ds}{X_s^2} \biggr\},
\]
which satisfies
\[
dN_t = r \frac{1}{X_t} N_t \, dX_t.
\]
The next lemma gives a precise statement about the
relationship.
%
\begin{lemma} There exists $c < \infty$ such that the following is
true. Suppose $\mu_1$
is the probability measure on paths given by
\[
X_t,\qquad 0 \leq t \leq1 \wedge T,
\]
where $X_t$ satisfies (\ref{may222alt}) with
$X_0 = x$ and
\[
T = \inf\{t\dvtx X_t \geq\pi/2\}.
\]
Let $\mu_2$ be the analogous probability measure using
(\ref{bessel}). Then
%
\begin{equation} \label{may283}
\frac1c \leq\frac{d\mu_1}{d\mu_2} \leq c.
\end{equation}
\end{lemma}
\begin{pf}From the explicit forms of $M_t,N_t$,
one can see there exists $c$ such that
\[
c^{-1} N_t \leq M_t
\leq c N_t,\qquad 0 \leq t \leq T \wedge1.
\]
\upqed\end{pf}

By comparison with the Bessel process, we
see that the process satisfying (\ref{may222alt})
never leaves $(0,\pi)$.
The invariant density for the process is
\[
h(x) = C_{2r} \sin^{2r} x,\qquad
C_{2r}^{-1} = \int_0^\pi\sin^{2r}y \, dy.\vadjust{\goodbreak}
\]
(This can be derived in a number of ways. It essentially
follows because the invariant density for Brownian motion
weighted by a function $F$ is proportional to $F^2$.)
In particular,
%
\begin{equation} \label{jun201}
\int_0^\pi h(x) [{\sin x}]^{1-2r} \, dx
= 2 {C_{2r}}.
\end{equation}
If $
t \geq0, x \in(0,\pi)$, we define
\[
\psi(t,x) = \E^x[(\sin X_t)^{1-2r}],
\]
where $X_t$ satisfies (\ref{may222alt}). Note that
$\psi(t,x) = \psi(t,\pi-x)$. Let
\[
\tilde\psi(t,x) = \E^x[X_t^{1-2r}],
\]
where $X_t$ satisfies (\ref{bessel}).
Using the previous lemma, we can see there exist $c_1,c_2$
such that
%
\begin{equation} \label{jun11}
c_1 \psi(t,x) \leq\tilde\psi(t,x)
\leq c_2 \psi(t,x),\qquad 0 < x \leq\pi/2,
0 \leq t \leq1.
\end{equation}
Let
\[
v(t,x) = \bigl[x \vee
\sqrt{t \wedge1}\bigr]^{1-4a}.
\]

The next lemma collects the facts about the SDE that we
will need.
%
\begin{lemma} \label{lemmamay22}
There
exists $c < \infty$
such that the following is true for all $x\in(0,\pi/2]$.
\begin{itemize}
\item If $t \geq1$,
%
\begin{equation} \label{jun10}
|\psi(t,x) - 2C_{2r} |
\leq c e^{- (r + 1/2) t}.
\end{equation}
\item If $t \geq0$,
%
\begin{equation} \label{jun12}
c^{-1} v(t,x) \leq
\psi(t,x) \leq c v(t,x).
\end{equation}
\item For every $\epsilon> 0$, there is a $\delta> 0$
such that if $t >0, 0 < x < \pi$, and $I \subset\R$
with
\[
\Prob^x\{X_t \in I\} \geq\epsilon,
\]
then
%
\begin{equation} \label{jun15}
\E^x[(\sin X_t)^{1-2r}; X_t \in I] \geq
\delta v(t,x).
\end{equation}
\end{itemize}
\end{lemma}
\begin{pf} The fact that
\[
\lim_{t \rightarrow\infty} \psi(t,x) = 2 C_{2r}
\]
follows from (\ref{jun201}).
To get the error estimate, one needs the next eigenvalue. Suppose $x_0
\leq\pi/2$ and let $T$
be the first time that the process reaches $\pi/2$. A simple
application of It\^o's formula shows that
\[
M_{t} = \cos(X_{t}) e^{(r+1/2) t},\vadjust{\goodbreak}
\]
is a martingale. Using this one can show that
\[
\Prob^{x_0}\{T \geq t\} \leq c e^{-(r+1/2) t}
\]
with a constant independent of the starting point and then
a coupling argument can be used to show that if we start two processes,
one at $x_0$ and the other in the invariant density, then the
probability that they do not couple by time $t$ is
$O(e^{-(r+1/2) t})$.

Note that (\ref{jun12}) for $t \geq1$ follows
immediately from (\ref{jun10}), so it suffices to
prove it for
$t \leq1$.

A coupling argument shows that if $x \leq x_1$,
then $\Prob^x\{X_t \leq y\}
\geq\Prob^{x_1}\{X_t \leq y\}$.
Therefore, if $X_0 = x$ and $h = C_{2r} (\sin x)^{2r}
$ denotes the invariant density,
\[
\Prob^x\{X_t \leq y\} \leq\frac{
\int_{0}^x \Prob^{s}\{X_t \leq y\} h(s)
\,ds } {\int_0^x h(s) \, ds}\leq
\frac{\int_0^y h(s) \, ds}
{\int_0^x h(s) \, ds}
\leq c (y/x)^{2r+1}.
\]
By integrating, we see for all $t$,
\[
\E^x[(\sin X_t)^{1-2r}] \leq
c x^{1-2r} + \E^x[(\sin X_t)^{1-2r}; X_t \leq x] \leq
c x^{1-2r}.
\]
This gives the upper bound for $t \leq x^2$;
for the lower bound for these $t$, we need only note that
there is a positive probability that the process starting
at $x$ stays within distance $x/2$ of $x$ up to time
$x^2$.

For the remainder, we consider $x^2 \leq t \leq1$.
The estimates
are more easily done for the Bessel process using scaling.
Since $t \leq1$, it suffices by (\ref{jun11}) to prove
that
\[
\tilde\psi(t,x) \asymp t
^{1/2 - r},\qquad
x^2 \leq t \leq1, 0 < x \leq
\pi/2,
\]
where the implicit constants in the $\asymp$ notation are
uniform over $t,x$. Suppose $X_t$ satisfies (\ref{bessel})
with $X_0 = x$. Then $X_t$ has the same distribution
as $x Y_{t/x^2}$, where $Y_t$ satisfies (\ref{bessel})
with $Y_0 = 1$. We consider $Y_t$ for $1 \leq t \leq
x^{-2}$.

To study this equation, it is convenient to consider
$Z_t = e^{-t/2} Y_{e^{t}}$. Note that
\begin{eqnarray*}
dZ_t & = & \biggl[-e^{-t/2} Y_{e^t} + e^{-t/2} \frac{r e^{t}}
{ Y_{e^{t}}}\biggr]
\, dt + e^{-t/2} \, dB_{e^{t/2}}\\
& = & \biggl[-Z_t + \frac{r}{Z_t} \biggr] \, dt + d\hat B_t,
\end{eqnarray*}
where $\hat B_t$ is a standard Brownian motion.
This is a positive recurrent SDE with invariant density
proportional to
\[
x^{2r} e^{-x^2}.
\]
We consider this with the initial condition
$Z_{-\infty}
= \infty$. (One can show this makes sense by writing the equation for
$R_t = 1/Z_t$ and showing that this can be well defined with
$R_{-\infty} = 0$.) By time $0$, the process is within
a constant multiple of the invariant density.\vadjust{\goodbreak}
All we will need from this are the following easily derived
facts.
\[
\E[Z_t^{1-2r}] \leq c,\qquad t \geq0.
\]
Secondly, for every $\epsilon> 0$, there is a $\delta
>0$ such that if $t \geq0$ and
\[
\Prob\{Z_t \geq K\} \geq\epsilon,
\]
then
\[
\E[Z_t^{1-2r}; Z_t \geq
K ] \geq\delta.
\]
\upqed\end{pf}

\subsubsection{The one-point estimate}

The results of the previous subsection will be used
with $r = 2a > 1/2$.
We use the radial parametrization to prove the
next proposition.
%
\begin{proposition} \label{junprop1}
If $z=x+iy$ and $\theta= \arg z$, then
for $\epsilon\leq y$,
\[
\Prob\{\tau_\epsilon< \infty\} = \epsilon^{2-d}
G(z) \psi(t,\theta),
\]
where $t$ satisfies $\epsilon= y e^{-2ta}$,
that is,
\[
\Prob\{\tau_\epsilon< \infty\} = \epsilon^{2-d}
G(z) \psi\biggl(\frac{\log(y/\epsilon)}{2a},\theta\biggr).
\]
\end{proposition}

It follows that for $1/2 \leq\epsilon\leq1$,
\[
\Prob\{ \tau_{\epsilon y}<\infty\} \asymp
\theta^{4a-1} \bigl[\theta\vee
\sqrt{1
-\epsilon}
\bigr]^{1-4a} = \min\bigl\{1,\bigl(\theta/\sqrt{1-\epsilon}\bigr)^{4a-1}
\bigr\}.
\]
For $\epsilon\leq1/2$,
\[
\Prob\{\tau_{\epsilon y} < \infty\} \asymp\epsilon^{2-d}
\theta^{4a-1}.
\]
\begin{pf*}{Proof of Proposition~\ref{junprop1}}
Let $\tau= \tau_{\epsilon}$. Since $M_{t \wedge\tau}$ is
a martingale, we have
\[
G(z) = \E[M_{t \wedge\tau} ]
= \E[M_t; t < \tau] + \E[M_\tau; \tau
\leq t ].
\]
We can let $t \rightarrow\infty$, and
from (\ref{may227}) and the dominated convergence
theorem, we can deduce that
\[
G(z) = \E[M_{\tau}; \tau< \infty].
\]
Also,
\begin{eqnarray*}
\Prob\{\tau< \infty\} &=&
\epsilon^{2-d} \E[\Upsilon_\tau^{d-2}; \tau< \infty]
= \epsilon^{2-d} \E[ M_\tau S_\tau^{1-4a}
; \tau< \infty]\\
&=& \epsilon^{2-d} G(z) \E^*[S_\tau^{1-4a}].
\end{eqnarray*}
From Girsanov, we see that
\[
\E^*[S_\tau^{1-4a}] = \psi(t,\theta).
\]
\upqed\end{pf*}

The following is a corollary of this and
Lemma~\ref{lemmamay22}.
%
\begin{proposition} \label{propmay}
For $\epsilon\leq y$,
%
\begin{equation} \label{may221}
\Prob\{\tau_\epsilon< \infty\} =
c_* G(z) \epsilon^{2-d}
\bigl[1 + O\bigl([{\epsilon}/y
]^{1 + 1/({4a})} \bigr) \bigr].\vadjust{\goodbreak}
\end{equation}
In particular,
there exist finite constants $c $ and $c'$ such that for all $z$ and
all~$\epsilon$,
%
\begin{equation} \label{may222}
\Prob\{\tau_\epsilon< \infty\} \leq c [\epsilon/\Im(z)]^{2-d}
\end{equation}
and
%
\begin{equation} \label{may263}
\Prob\{\tau_\epsilon< \infty\} \geq
c' G(z) \epsilon^{2-d}.
\end{equation}
\end{proposition}

\subsection{\texorpdfstring{Two-sided chordal $\mathit{SLE}_\kappa$}
{Two-sided chordal SLE kappa}}\label{twochordal}

When $z = x+iy$ with $|x| \gg y$, then the radial parametrization
is not so useful, because the path can get close to $z$ without
significantly decreasing the radial parametrization. Fortunately,
one can study such processes by studying another process which
we call \textit{two-sided chordal $\mathit{SLE}_\kappa$ through~$x$}.
It is $\mathit{SLE}_\kappa$ weighted locally by $X_t^{1-4a}$ (as
compared to $|Z_t|^{1-4a}$ for two-sided radial).

Let us fix $z=x+iy$ and for convenience we assume
$x > 0$. We write $\tilde X_t = Z_t(x) = g_t(x) - U_t$
which satisfies
\[
d\tilde X_t =\frac{a}{\tilde
X_t} \, dt + dB_t.
\]
It\^o's formula shows that
\[
d \tilde X_t^{1-4a} = X_t^{1-4a}
\biggl[\frac{a(4a-1)}{X_t^2} \,dt
+ \frac{1-4a}{X_t} \, dB_t
\biggr]
\]
and hence
\[
N_t = \tilde
X_t^{1-4a} \exp\biggl\{-a(4a-1)\int_0^t
\frac{ds}{\tilde X_s^2} \biggr\}
= \tilde X_t^{1-4a} g_t'(x)^{4a-1}
\]
is a local martingale satisfying
\[
dN_t = \frac{1-4a}{\tilde X_t} M_t
\, dB_t.
\]
Hence,
\[
dB_t = \frac{1-4a}{\tilde X_t} \, dt +
d\tilde W_t,
\]
where $\tilde W$ is a Brownian motion in the weighted measure.
In particular,
\[
d \tilde X_t = \frac{1-3a}{\tilde X_t} \, dt
+ d\tilde W_t.
\]
From this, we see that the terminal time $T_x$ is finite
with probability one in the weighted measure.
Proposition~\ref{jun8prop1}
shows that $\gamma(T_x) = x$.
(The proof is similar to the proof that $\mathit{SLE}_\kappa$ hits
points for $\kappa\geq8$.) We precede this with a
standard lemma
about Bessel process whose proof we omit.
%
\begin{lemma} \label{jun8lemma}
Suppose $r < 1/2$ and $X_t$ satisfies the
Bessel equation
\[
dX_t = \frac{r}{X_t} \, dt + d B_t,\qquad
X_0 = 1.\vadjust{\goodbreak}
\]
Let $T = \inf\{t\dvtx X_t = 0\}$. Then with probability one,
\[
T < \infty,\qquad
\int_0^T \frac{dt}{X_t} < \infty,\qquad
\int_0^T \frac{dt}{X_t^2} = \infty.
\]
Moreover, for every $\delta> 0$,
%
\begin{equation} \label{jun8111}
\Prob\biggl\{ 0 \leq X_t \leq1 + \delta, 0 \leq t \leq\delta;
T \leq\delta; \int_0^T \frac{dt}{X_t} \leq
\delta\biggr\}
> 0.
\end{equation}
\end{lemma}

If $B_0 = 0$, then
\[
-B_t = X_0 - X_t + \int_0^t \frac{r}{X_s} \, ds
= 1 - X_t + \int_0^t \frac{r}{X_s} \, ds
\]
and hence on the event described in (\ref{jun8111})
\[
- \delta\leq- B_t \leq1 + \delta,\qquad
0 \leq t \leq T.
\]

\begin{proposition} \label{jun8prop1} Suppose $0 < x < x_1$ and let
\mbox{$R_t = Z_t(x_1) = g_t(x_1) - U_t$}. Then with probability one,
if $\gamma$ is two-sided chordal to $x$ with terminal
time $T = T_x$,
\[
R_{T} > 0.
\]
\end{proposition}
\begin{pf} Since $1-3a < \frac12$, Lemma
\ref{jun8lemma} implies that
with probability one $T < \infty$ and
%
\begin{eqnarray}
\label{jun41}
\int_0^T \frac{dt}{\tilde X_t} &<& \infty,
\\
%
\label{jun42}
\int_0^T \frac{dt}{\tilde X_t^2} &=& \infty.
\end{eqnarray}

Let $q(x,y)$ denote the probability that $R_T > 0$ given
$\tilde X_0 = x, R_0 = y$. By Bessel scaling, $q(x,y) =
q(1,y/x)$.
We claim that $q(1,r) \rightarrow1$ as $r \rightarrow
\infty$. Indeed, we have
\begin{eqnarray*}
\log R_T &=& \log[R_T - \tilde X_T] =
\log(y-x) + \int_0^T \p_t[ \log(R_t -
\tilde X_t)] \, dt\\
&=&\log(y-x) - \int_0^T \frac{a \, dt}{\tilde X_t R_t}.
\end{eqnarray*}
From this and (\ref{jun41}), we
can deduce this fact. From this and the strong
Markov property, it follows that on the event
%
\begin{equation} \label{jun43}
\sup_{0 \leq t < T} \frac{R_t -
\tilde X_t}{\tilde X_t} =
\infty,
\end{equation}
we must have $R_T > 0$.\vadjust{\goodbreak}

If
\[
L_t = \log\frac{R_t - \tilde X_t}{\tilde X_t},
\]
then
\[
dL_t = \biggl[
- \frac{a}{\tilde X_t R_t}
-\frac{ 1/2-3a}{\tilde X_t^2}\biggr] \, dt
- \frac1{\tilde X_t} \, d\tilde W_t.
\]
Under a suitable time change $\hat L_t = L_{\sigma(t)}$,
this can be written as
\[
d\hat L_t = \biggl[
- \frac{a \hat X_t}{\hat R_t} - \biggl(\frac
12 - 3a\biggr)\biggr] \, dt
+ d\hat B_t
\]
for a standard Brownian motion $\hat B_t$. By
(\ref{jun42}), in this time change it takes infinite time
for $\hat X_t$ to reach zero, that is, $\sigma(\infty) = T$.
Since $\hat X_t \leq\hat R_t$, the drift term
is bounded below by
\[
-a -\bigl(\tfrac12 - 3a\bigr) > 0
\]
and hence $\hat L_t \rightarrow\infty$ as $t \rightarrow\infty$.
In particular, (\ref{jun43}) holds.
\end{pf}

If $0 < x < x_1 < x_2$, then the random variable
\[
\Delta=\Delta(x,x_1,x_2) =
\max_{0 \leq t \leq T_x} \frac{Z_{t}(x_2)}
{Z_{t}(x_1) }
\]
is well defined and satisfies $1 < \Delta< \infty$.
Moreover, Bessel scaling implies that the distribution
of $\Delta(x,x_1,x_2)$ is the same as that
of $\Delta(rx,rx_1,rx_2)$. The next lemma
gives uniform bounds on the distribution in
terms of $(x_2-x)/(x_1-x)$.
%
\begin{lemma} For every $\rho> 0$ there exists
$C < \infty$ such that the following is true. Suppose
$x < x_1 < x_2$ with $x_1 - x \geq\rho(x_2-x)$.
Then with probability at least $1-\rho$,
$\Delta(x,x_1,x_2) \leq C$.
\end{lemma}
\begin{pf} We fix $\rho$ and allow constants to depend
on $\rho$. By scaling and monotonicity, we may assume
$x=1$ and $x_1 - x = \rho(x_2-x)$. Hence, we write
$x_1 = 1 + s\rho, x_2 = 1 + s$ where $s > 0$ (our estimates
must be uniform over $s$). If $s$ is very large,
then since $T_x < \infty$, with probability one, one can
find a $K$ such that with probability at least $1-\rho$
for all $\tilde x \geq K\rho$.
\[
\frac{Z_0(\tilde x)}{2}
\leq Z_t(\tilde x) \leq2 Z_0(\tilde
x),\qquad 0 \leq t \leq T_x,
\]
and hence with probability at least $1-\rho$,
$\Delta(1,1+s\rho,1 + s) \leq4$. For
$1/2 \leq s \leq K$, we can bound
\[
\Delta(1,1+s\rho,1+s) \leq
\Delta\biggl(1,1 + \frac\rho2,
1+ K\biggr).
\]

For the remainder, we assume that $s \leq1/2$.
Let
\[
\xi= \inf\{t\dvtx|\gamma(t) - 1| = 2s\}.
\]
By distortion estimates, we see that
\[
g_t'(x) \asymp g_t'(x_1) \asymp
g_t'(x_2),\qquad 0 \leq t \leq\xi,
\]
with the implicit constants uniform over $s$. Since
$g_t'(x')$ is an increasing function of $x'$, this implies
\[
Z_t(x_1) \asymp Z_t(x_2),\qquad
0 \leq t \leq\xi,
\]
which gives uniform estimates (with probability
one) on
\[
\max_{0 \leq t \leq\xi} \frac{Z_{t}(x_2)}
{Z_{t}(x_1) }.
\]
Also,
using the $1/4$-theorem with distortion estimates,
\[
Z_{\xi}(x) \asymp Z_{\xi}(x_1) - Z_\xi(x)
\asymp Z_{\xi}(x_2) - Z_\xi(x)
\]
with again the estimates uniform over $s$. The conditional
distribution of
\[
\max_{\xi\leq t \leq T_x} \frac{Z_{t}(x_2)}
{Z_{t}(x_1) }
\]
given $\F_\xi$ is the distribution of $\Delta(Z_\xi(x),
\Z_\xi(x_1), Z_\xi(x_2))$.
Hence, we reduce this to the case where $x,x_1,x_2$ are comparable
which we have already stated and handled.
\end{pf}
%
\begin{lemma} \label{jun8lem3}
For every $\rho> 0$, there is
a $u > 0$ such that the following holds. Suppose
$x > 0$ and $\gamma$ is two-sided
chordal $\mathit{SLE}_\kappa$
to $x$. Let $ 0< \epsilon
\leq x$ and
\[
\xi= \inf\{t\dvtx |\gamma(t) - x| = \epsilon\}.
\]
Define $\psi$ by
\[
\gamma(\xi) = x + \epsilon e^{i\psi}.
\]
Then with probability at least $1-\rho$,
$\psi\geq u$.
\end{lemma}
\begin{pf} As $\psi\to0$, the extremal distance between
$(x,x+\epsilon/2)$ and $(x+2\epsilon,\infty)$
in $H_{\xi}$ tends to $\infty$. By conformal invariance of extremal
distance, if $\psi$ is very close to zero, then we can
see that $Z_{\xi}(x+2 \epsilon) - Z_\xi(x) \gg
Z_\xi(x + \frac\epsilon
2) - Z_\xi(x)$. The previous lemma bounds the probability that
this happens.
\end{pf}

\subsubsection{Comparison with two-sided radial}

In this section, we fix $ x > 0$
$\epsilon< x/2$. If $z \in\Half$,
we write
$Z_t = Z_t(z)= X_t + iY_t$, $\tilde X_t = Z_t(x)$ as above.
Recall that two-sided radial to $z$ is the process obtained
by weighting with respect to the local martingale
\[
M_t = |Z_t|^{1-4a} (\Upsilon_t/y)^{1/({4a}) - 1}
(Y_t/y)^{4a-1}.\vadjust{\goodbreak}
\]
Here we have included a constant $y^{2-4a-1/({4a})}$ to the usual local
martingale, but this does not affect the weighting. Under this
choice of $M_t$, $M_0 = |z|^{1-4a}$. Two-sided chordal
is obtained by weighting by the local martingale
\[
N_t = |\tilde X_t|^{1-4a} g_t'(x)^{4a-1}.
\]

Let
\[
\sigma= \inf\{t\dvtx|\gamma(t) - x| = 2 \epsilon\}
.
\]
Distortion estimates and the $1/4$-theorem imply
that for $0 \leq t \leq\sigma$ and $|z-x| \leq\epsilon$,
\[
g_t'(x) \asymp|g_t'(z)|,\qquad
\Upsilon_t \asymp\Upsilon_0 = y,\qquad
Y_t \asymp y |g_t'(z)|,\qquad
\tilde X_t \asymp|Z_t|
\]
and hence
\[
M_t \asymp N_t,\qquad
0 \leq t \leq\sigma,
\]
where the implicit constants are uniform over $z$.
We have just proved the following.
%
\begin{proposition} \label{rn}
There exists $c < \infty$, such that
for every $z=x+iy$ with the property $|z-x|\leq\epsilon$, if
$\mu$ denotes the measure on paths
\[
\gamma(t),\qquad 0 \leq t \leq\sigma,
\]
given by two-sided radial $\mathit{SLE}_\kappa$ to $z$ and
$\nu$ denotes the analogous measure on paths using two-sided
chordal to $x$, then
\[
c^{-1} \leq\frac{d\mu}{d\nu} \leq c.
\]
\end{proposition}

The following is a corollary of this proposition and
Lemma~\ref{jun8lem3}.
%
\begin{lemma} \label{jun8lem4}
For every $\rho> 0$, there is a $u > 0$ such that the following
is true. Suppose $x > 0$ and $\delta< x$. Suppose
$z \in\Half$ with $|z-x| \leq\delta/2$. Let $\gamma$ be a two-sided
radial $\mathit{SLE}_\kappa$ path through $z$, and let
\[
\sigma= \inf\{t\dvtx|\gamma(t) - x | = \delta\}.
\]
Then,
\[
\Prob^*_z\{ S_\xi(z) \geq u\} \geq1 - \rho.
\]
\end{lemma}

\section{\texorpdfstring{Proof of (\protect\ref{may274})}
{Proof of (14)}} \label{Section3}
In Section~\ref{importantsec}, we spend a relatively long
time deriving an estimate that states roughly that there is a
positive probability that a two-sided radial path stays
with a small distance of an ``$L$''-shape.
The proof uses
two well-known ideas:
\begin{itemize}
\item If $f\dvtx[0,1] \rightarrow\R$ is a continuous
function with $f(0) = 0$ and $\epsilon> 0$, and
$U_t$ is a standard Brownian motion, then the
probability that $\|f - U\|_\infty< \epsilon$
is positive, where $\| \cdot\|_\infty$ denotes the supremum norm.\vadjust{\goodbreak}
\item If the driving functions (using the Loewner equation)
of two curves
are close in the supremum norm, then the curves are close
in the Hausdorff metric. This follows in a straightforward manner
from the Loewner equation. (It is possible that the curves are
not close in the supremum norm on curves, but this is not
important for us.)
\end{itemize}
It takes a little time to write out the details because our driving
functions are those for two-sided radial and we need some
uniformity in the estimates. Readers who are convinced
that this can be done can skip the proof of Proposition
\ref{hullprop}. Proposition~\ref{determine} is a simple deterministic
estimate which uses Proposition~\ref{hullprop} to establish
(\ref{may274}) for $z,w$ far apart.

It remains to prove (\ref{may274}) for $z,w$ close and
this is the goal of Section~\ref{mainestimate}. Indeed, this
does not seem that it should be difficult, since the events
of getting close to $z$ and getting close to $w$ should be
very positively correlated.
We give the argument in two
cases: when $|z-w|$ is much smaller than $\Im(z)$ and when
$z,w$ are close to each other and also close to the boundary.

\subsection{Probability of $L$-shapes}
\label{importantsec}

If $z = x+iy$, let $L_z$ denote the ``$L$''-shape
\begin{eqnarray*}
L_z &=& [0,x] \cup[x,x+iy],\qquad x \geq0,
\\
L_z &=& [x,0] \cup[x,x+iy],\qquad x \leq0,
\end{eqnarray*}
and if $\rho>0$,
\[
L_{z,\rho} = \{z' \in\overline\Half\dvtx \dist
(z',L_z) \leq\rho|z|\}.
\]
The goal of this section is to prove the following proposition.
%
\begin{proposition} \label{hullprop}
For every $\rho> 0$, there is a $u > 0$ such that
for all $z = x+iy \in\Half$,
\[
\Prob^*_z\{
\gamma[0,T_z] \subset L_{z,\rho}\}
\geq u.
\]
\end{proposition}

The important thing is
to show that we can choose $u$ uniformly over $z$.
Before discussing the proof of this proposition, let us
show how it can be used to prove (\ref{may274}) for $z$
and $w$ sufficiently spread apart.
We use the following
deterministic estimate.
%
\begin{proposition} \label{determine}
For every $0 < \rho\leq1/4$ there exists
$c < \infty$ such that the following holds. Suppose
$z\in\Half$, $w \in\Half
\setminus L_{z,2\rho}$, and suppose
$\gamma\dvtx[0,T] \rightarrow\overline\Half$ is a curve
with $\gamma(0)=0, \gamma(T) = z$ and $\gamma[0,T]
\subset L_{z,\rho}$.
Let $g = g_T$ be the corresponding
conformal transformation and let $Z = g(w) - g(z)$. Then
\begin{eqnarray*}
G(Z) &\geq& c G(w),
\\
c^{-1} &\leq&|g'(w)| \leq c.
\end{eqnarray*}
\end{proposition}
\begin{pf} By scaling it suffices to consider
$|z| = 1$ which we assume. The estimate is easy for $|w| \geq2$
so we assume\vadjust{\goodbreak} $w \in\Half
\setminus L_{z,2\rho}$ with $|w| \leq2$. We fix $\rho$ and allow
(implicit or explicit)
constants to depend on $\rho$.
Let $ J = \{w'\dvtx\Im(w') \geq3\}$.
Using conformal invariance,
one can see that $\Im[g(w)]$ is comparable to the probability
that a Brownian motion
starting at $w$ reaches $J$ without
hitting $\R\cup\gamma[0,T]$ and this is not smaller
than the the probability of reaching $J$ before hitting
$\R\cup L_{z,\rho}$. It is not difficult to see that this is
bounded below by a constant times $\Im(w)$ and hence
\[
c \Im[w] \leq\Im[g(w)] \leq\Im[w].
\]
Using (\ref{jun91}), we also see that
\[
S_{H_T}(w) \asymp\Im[g(w)] \asymp\Im[w]
\asymp S_0(w)
\]
(remember that the implicit constants depend on $\rho$).
Also, $\dist(w,\R\cup L_{z,\rho}) \asymp\Im(w)$, and hence
by the Koebe-$1/4$ theorem,
\[
|g'(w)| \asymp1.
\]
Therefore, $G(Z) \asymp G(w)$.
\end{pf}
%
\begin{corollary} \label{may28cor1}
For every $\rho> 0$, there
exists $c > 0$ such that if $z \in\Half$ and $
w \in\Half\setminus L_{z,\rho}$, then
$ F(z,w) \geq c. $
\end{corollary}
\begin{pf}
By Propositions~\ref{hullprop} and~\ref{determine}, there
exists $u = u(\rho) > 0$ such that
\[
\Prob_z^*\{G_{H_T}(w;z,\infty) \geq
u G(w) \} \geq u.
\]
Therefore,
\[
\E_z^*[G_{H_T}(w;z,\infty)]
\geq u^2 G(w).
\]
\upqed\end{pf}

The proof of Proposition~\ref{hullprop} combines a probabilistic
estimate for the driving function and a deterministic estimate
using the Loewner equation. Assume $z=x+iy, x\geq0, y > 0$.
Note that one obtains the hull
$[x,x+iy]$ by solving (\ref{loeweq}) with $U_t \equiv x,
0 \leq t \leq y^2/(2a)$.
For every $\delta> 0$, let $ E_{z,\delta}$ denote
the event
\[
E_{z,\delta}
= \biggl\{ T_z \leq\frac{y^2}{2a} + \delta;
-\delta\leq U_t \leq x + \delta, 0 \leq t \leq\delta;
|U_t - x| \leq\delta, \delta\leq t \leq T_z \biggr\}.
\]

\begin{lemma} \label{newdetermine}
For every $\rho> 0$, there exists $\delta> 0$
such that the following holds. Suppose
$z= x+iy$ with $0 \leq x \leq1, 0 < y \leq1$. Then on the event
$E_{z,\delta}$,
\[
\gamma[0,T_z] \subset L_{z,\rho}.
\]
\end{lemma}
\begin{pf} This is a straightforward deterministic estimate
using (\ref{loeweq}). One first shows that if $\delta$ is sufficiently
small, then $|g_\delta(w) - w| \leq\rho/100$ for $\dist(w,\break[0, x])
\geq\rho$. For\vspace*{1pt} $\delta\leq t \leq T_z$, we compare $g_t$ to the
corresponding function $\tilde g_t$
obtained with $\tilde U_t \equiv x$. These estimates are standard
(see, e.g.,~\cite{lawlerbook}, Proposition~4.47), and we omit
the details.\vadjust{\goodbreak}
\end{pf}

Therefore, Proposition~\ref{hullprop} reduces to the following
probabilistic estimate on the driving function.
%
\begin{lemma} \label{partytime}
For every $\delta> 0$, there exists $u > 0$
such that if $z= x+iy$ with $0 \leq x \leq1, 0 < y \leq1$,
then
\[
\Prob_z^*[E_{z,\delta}] > u.
\]
\end{lemma}

We will use the next lemma in the proof
of Lemma~\ref{partytime}. The lemma may
seem to follow immediately from
$T_z < \infty$, but it is
important to establish uniformity in~$z$.
%
\begin{lemma} \label{anotherlemma}
For every $\epsilon\,{>}\, 0$, there exists $r \,{<}\, \infty$ such
that if \mbox{$z \,{=}\, x \,{+}\,iy\,{\in}\,\Half$},
%
\begin{equation} \label{jun85}
\Prob^*_z\{T_z \leq r|z|^2; |U_t| \leq r|z|, 0 \leq t
\leq
T_z \}
\geq1- \epsilon.
\end{equation}
\end{lemma}
\begin{pf}
By scaling and symmetry, it suffices to prove the result
for $|z| = 1$, $x \geq0$. For fixed $z$, the result is immediate
from the fact that $\Prob_z^*\{T_z < \infty\} = 1$. It is not
difficult to extend this as follows: for every $u > 0$ and
every $\epsilon> 0$, there exists $r$ such that if $|z| = 1$
and $S_0(z) \geq u$, then
%
\begin{equation} \label{jun87}
\Prob^*_z\{T_z \leq r; |U_t| \leq r, 0 \leq t \leq
T_z \}\geq1-\epsilon.
\end{equation}

For $y < 1/100$, let $\Prob_1^*$ denote probabilities for two-sided
chordal $\mathit{SLE}_\kappa$
to~$1$. Again, since $\Prob_1^*\{T_1 < \infty\}=1$, it
is easy to see that for every $r$, there exists $\epsilon> 0$
such that
%
\begin{equation} \label{jun86}
\Prob^*_1\{T_1 \leq r; |U_t| \leq r, 0 \leq t \leq
T_1 \} \geq1 - \epsilon.
\end{equation}
Suppose $z = x + iy \in\Half$ with $|z-1| = \delta/2 < 1/4$. Let
\[
\sigma= \inf\{t\dvtx|\gamma(t) - 1| = \delta\}
\]
and define $\xi$ by $\gamma(\sigma) = 1 + \delta e^{i\xi}$.
From (\ref{jun86})
and Proposition~\ref{rn}, we see that there exists $c_1$ such
that
\[
\Prob^*_z\{\sigma\leq r; |U_t| \leq r, 0 \leq t \leq
\sigma\} \geq1 - c_1\epsilon.
\]
Using Lemma~\ref{jun8lem4}, we see that there exists $u > 0$, such that
\[
\Prob^*_z\{\sigma\leq r; |U_t| \leq r, 0 \leq t \leq
\sigma; S_\sigma(z) \geq u \} \geq1 - 2c_1\epsilon.
\]
Using the form of the Poisson kernel in the upper half plane, we can
see that $|Z_\sigma(z)| \leq c_2 \delta$. By using (\ref{jun87}),
we see that there exists $\tilde r$ such that
\begin{eqnarray*}
&&\Prob^*_z\{\sigma\leq r; |U_t| \leq r, 0 \leq t \leq
\sigma; S_\sigma(z) \geq u;
T_z - \sigma\leq\tilde r \delta^2;
\\
&&\hspace*{125pt}|U_t - U_\sigma| \leq\tilde r \delta, \sigma\leq t \leq
T_z \} \\
&&\qquad\geq1 - 3c_1\epsilon.
\end{eqnarray*}
\upqed\end{pf}
\begin{pf*}{Proof of Lemma~\ref{partytime}}
By scaling and symmetry, we may assume that $z=x+iy$ with
$|z| = 1$ and $x \geq0, y > 0$.
We
will show that there exists\vadjust{\goodbreak} $c < \infty$ such that
for each $\rho$, there exists $q(\rho) > 0$ such
that for all $z$,
\[
\Prob_z^* \bigl[E_{z,c\sqrt\rho}\bigr] \geq q(\rho).
\]
If suffices to consider $0 < \rho\leq1/1000$.
We will consider two cases: $y \leq10\rho$ and $y > 10\rho$.
In this proof, constants $c_1,c_2,\ldots$ are independent of
$\rho$, but constants $\delta, q_1,q_2,\ldots$ may depend
on $\rho$.

First assume $y \leq10 \rho\leq1/100$, and hence
$3/4 < x \leq1$.
Let
\[
\eta= \eta_{\rho,z} = \inf\{t\dvtx\Re[\gamma(t)] = x - 4 \rho\}.
\]
For every $\delta> 0$, consider the event $V_\delta= V_{\delta,\rho,z}$
given by
\[
V_\delta= \{\eta\leq\delta; -\delta\leq U_t \leq1 + \delta,
0 \leq t \leq\eta\}.
\]
Using the deterministic estimate, Lemma~\ref{newdetermine},
we can see that by choosing $\delta$ sufficiently small, then on
the event $V_\delta$, $\Im[\gamma(\eta)] \leq\rho$. By choosing
$\delta$ smaller if necessary, we
assume $\delta< \rho$.

Using Lemma~\ref{jun8lemma}, we can see that $\Prob[V_\delta] \geq
q_1 > 0$. There is a curve
of length at most $11 \rho$ in $H_\eta$ connecting $z$ and $\gamma
(\eta)$.
Hence, using the Beurling estimate,
there exists $c_1$ such
that
\[
|Z_\eta(z)| \leq c_1 \sqrt\rho.
\]
[Actually, we can get an estimate of $O(\rho)$, but the estimate
above suffices for our purposes.]
Using Lemma~\ref{anotherlemma}, we can say there exists $c_2$ such that
\[
\Prob\Bigl\{ T - \eta\leq c_2 \sqrt\rho,
\sup_{\eta\leq t \leq T}
|U_t - U_\eta| \leq c_2 \sqrt\rho{\big|}
V_\delta\Bigr\} \geq\frac12.
\]
Therefore, with probability at least $q_1/2$,
\[
T \leq(c_2 + 1) \sqrt\rho,\qquad
-(1+c_2) \sqrt\rho\leq U_t \leq1 + (1 + c_2) \sqrt\rho.
\]

We now assume $y \geq10\rho$.
Let
\[
\eta= \eta_{\rho,z} =
\inf\{t\dvtx\Im[\gamma(t)] = y - 4 \rho\}.
\]

Let $W_t$ denote a standard Brownian motion and consider the event
$E = E_{\delta,x}$ that
\begin{eqnarray*}
|W_t - (tx/\delta)| &<& \delta,\qquad 0 \leq t \leq\delta,
\\
|W_t - x| &<& \delta,\qquad \delta\leq t \leq1/a.
\end{eqnarray*}
Using standard estimates for Brownian motion (including the
Cameron--Martin formula), it is standard to show that for every $\delta
> 0$
there exists $u_1 > 0$ such that for all $0 \leq x \leq1$,
$\Prob(E) \geq u_1$. If we let $U_t = W_t$, then by choosing $\delta$
sufficiently small, we see that
$\Prob[E] \geq u_1. $

We claim that there exists $c_3 > 0$ such that
on the event $E$,
\[
S_\eta(z) \geq c_3.
\]
To show this, we consider the path $\gamma[0,\eta]$. Let
$\gamma^+$ be the part of the path mapped to $[U_\eta,\infty)$
under\vadjust{\goodbreak} $g_\eta$ and let $\gamma^-$ be the part mapped to $(-\infty
,U_\eta)$.
Using the fact that $\gamma[0,\eta] \subset L_\rho$; $\Im[\gamma
(\eta)] = y-4 \rho$
and $\Im[\gamma(t)| < y - 4 \rho, t < \eta$, we can see geometrically
that there is a positive probability $u_2$ such that a Brownian motion
starting at $z$ exists $H_\eta$ at $\gamma^+$ with probability at
least $c_2'$ and at $\gamma^-$ with probability at least $c_2'$.
This combined with
(\ref{jun91}) gives the lower bound on $S_\eta(z)$.
Since $\Upsilon_t(z)$ decreases with $t$, we get a lower bound
on $M_\eta(z)$. Therefore, there exists $q_2 > 0$ such that
\[
\Prob_z^*
\{\gamma[0,\eta] \subset L_{z,\rho}\}
\geq q_2.
\]
[The reader may note that we have used the trivial bound
$\Upsilon_\eta(z) \leq1$. In fact, $\Upsilon_\eta(z) \asymp
\rho$, so we can improve the last estimate but we do not need to.
For $\rho$ small, it is much
more likely for two-sided radial $\mathit{SLE}$ to follow
the $L$-shape to $z$ then for usual $\mathit{SLE}$.]

On the event $E$, there is a curve connecting $\gamma(\eta)$ to $z$
in $H_\eta$ of length $O(\rho)$. Using the Beurling estimate,
we can see that there exists $c_4$ such that $|Z_\eta(z)|
\leq c_4 \sqrt\rho$. The proof proceeds as in the previous
case.
\end{pf*}

\subsection{Remainder of proof} \label{mainestimate}

To finish the proof, we need to consider $z,w$ that are close.
In this case, we will take a stopping time $\sigma$ such that
$z,w$ are not so close in the domain $H_\sigma$. The next lemma
is easy, but it is useful to state it.
%
\begin{lemma} \label{jun1lemma1}
There exists $c > 0$ such that the following is true.
Suppose $z$, $w \in\Half, u \geq0$ and $\sigma$ is a stopping time
for $\mathit{SLE}_\kappa$ such that
\[
|\gamma(t) - z| \geq3 |z-w|,\qquad
0 \leq t \leq\sigma,
\]
and such that
%
\begin{equation} \label{jun92}
\Prob_z^*\{
F_{H_\sigma}(z,w;\gamma(\sigma),\infty)
\geq u\} \geq u.
\end{equation}
Then
\[
F(z,w) \geq
c u^3.
\]
\end{lemma}
\begin{remark*} Implicit in the assumptions is $|z| \geq3|z-w|$.
We do not assume that the disk of radius $|z-w|$ about $z$ is contained
in $\Half$. In particular, $\Im(z),\Im(w)$ can be small and
very different.
\end{remark*}
\begin{pf*}{Proof of Lemma~\ref{jun1lemma1}}
Let $r = |z-w|$ and let $\ball$ denote the open disk of radius
$2r$ centered at $z$. By assumption $0 \notin\ball$. It
is not necessarily the case that $\ball\subset\Half$; however,
for $t \leq\sigma$, the conformal map $g_t$ can be extended
to $\ball$ by Schwarz reflection.

Let $E = E_u$ denote the event that $F_{H_\sigma}(z,w;\gamma(\sigma
),\infty)
\geq u$.
Using the Beurling estimate, we can see that $M_{t \wedge\sigma}(z),
M_{t \wedge\sigma}(w)$ are bounded martingales and hence
\[
G(z) = M_0(z) = \E[M_\sigma(z)],\qquad
G(w) = M_0(w) = \E[M_\sigma(w)].
\]
Also, by definition of $ \E^*_z$,
\begin{eqnarray*}
\E^*_z[M_\sigma(z) 1_E]
&=& \frac{\E[M_\sigma(z)^2 1_E]}{M_0(z)}
= \frac{\E[M_\sigma(z)^2 1_E]}{ \E[M_\sigma
(z) ]
} \geq u \frac{\E[M_\sigma(z)^2 1_E]}{ \E
[M_\sigma(z)
1_E ]
}
\\
&\geq& u \E[M_\sigma(z) 1_E ]
\geq u^2 G(z).
\end{eqnarray*}
The first inequality uses (\ref{jun92}).
Using the distortion theorem on $\ball$, we can see that
$|g_\sigma'(z)| \asymp|g_\sigma'(w)|$.
We also claim that
%
\begin{equation} \label{jun93}
\frac{S_\sigma(z)}{S_\sigma(w)} \asymp
\frac{\Im(z)}{\Im(w)}.
\end{equation}
To see this, consider the first time that a Brownian motion
starting at $z,w$ reaches $\R\cup\p\ball$. If $p(z),p(w)$
denotes the probabilities that the process hits $\p\ball
\cap\Half$ before leaving $\Half$, then standard estimates
(gambler's ruin estimate) show that $p(z)/p(w) \asymp
\Im(z)/\Im(w)$. Also, the conditional distributions given
that one hits $\p\ball$ are mutually absolutely
continuous (here we use either a boundary Harnack principle
or the explicit form of the Poisson kernel in a half disk).
Given this and~(\ref{jun91}), we can conclude (\ref{jun93}).

Therefore, using
(\ref{greenscale}), we see that
\[
\frac{M_\sigma(z)}{M_\sigma(w)}
= \frac{G_{H_\sigma}(z;\gamma(\sigma),\infty)}
{G_{H_\sigma}(w;\gamma(\sigma),\infty)} \asymp
\frac{G(z)}{G(w)}
\]
and hence,
\[
\E^*_z[M_\sigma(w) 1_E]
\geq c u^2 G(w).
\]
Also,
\[
\E_z^*[G_{H_{T_z}}(w;z,\infty) 1_E |\F_\sigma]
= 1_E u G_{H_\sigma}(w;\gamma(\sigma),\infty)
= 1_E u M_\sigma(w).
\]
Taking expectations, we get
\[
\E_z^*[G_{H_{T_z}}(w;z,\infty)]
\geq\E_z^*[G_{H_{T_z}}(w;z,\infty) 1_E]
\geq c u^3 G(w),
\]
which implies $F(z,w) \geq c u^3$.
\end{pf*}
\begin{pf*}{Proof of (\ref{may274})}
By symmetry and scaling, it suffices to consider
$1 = |z| \leq|w|$ with $\Re(z) \geq0$. If $|w| >
1.01$, then $w \notin L_{z,1/100}$ and hence Corollary~\ref{may28cor1}
implies that $F(z,w) \geq c >0$.

Similarly, if $\Im(z) > \Im(w) + (1/100)$, then
$z \notin L_{w,1/100}$, and Corollary~\ref{may28cor1}
implies that $F(w,z) \geq c$.
Using similar facts about real parts
and interchanging~$z,w$, it suffices to consider $z=x+iy,
w = \tilde x + i \tilde y$
with $1 \leq|z|,|w| \leq1.01$,
$x \geq0$ and
\[
y \leq\tilde y \leq y + \frac1{100}
,\qquad |x - \tilde x | \leq\frac1{100}.
\]
We now split into two cases: $\tilde y \geq1/10$ and
$\tilde y < 1/10$.\vadjust{\goodbreak}

For $\tilde y \geq1/10$, let $\tau= \inf\{t\dvtx\Upsilon_t(z)
= 10|z-w|\}, T = T_z$.
Using Lemma~\ref{lemmamay22}, we see that there exists $u > 0$ such that
\[
\Prob_z^*\{S_\tau(z) \geq1/4\} \geq u.
\]
Using distortion estimates and Corollary~\ref{may28cor1}, we can see that
there exists $c > 0$ such that on the event
that $S_\tau(z) \geq1/4$,
\[
F_{H_\tau}(z,w;\gamma(\tau), \infty) \geq c.
\]
We can now apply Lemma~\ref{jun1lemma1}.

If $\tilde y < 1/10$ and $|z-w| \leq\tilde y/20$, we can do similarly
as above, interchanging the roles of $z$ and $w$.

For the remainder, we assume that
$\tilde y < 1/10$ and $|z-w| \geq\tilde y/20$. Note that
$x,\tilde x > 9/10$. Let
\[
r = \max\{\tilde y, |z-w| \} < 1/10.
\]
Let $\tau= \inf\{t\dvtx|\gamma(t) - \tilde x| = 4r\}$.
If we write
\[
\gamma(\tau) = \tilde x + 4r e^{i\xi},
\]
then by Lemma~\ref{jun8lem4}
there exists $u$ such that
\[
\Prob_z^*\{ \xi> u \} \geq u.
\]
On this event, distortion estimates
and Corollary~\ref{may28cor1} imply that
$F_{H_\tau}(z,w;\break\gamma(\tau),\infty) \geq c$ for
some $c$ (depending on $\xi$). We can now apply Lemma~\ref{jun1lemma1}.~%
\end{pf*}

\section*{Acknowledgment}

The authors would like to thank a referee for his/her valuable
comments, which improve the exposition of this work.



%
\printaddresses


\begin{thebibliography}{19}

\bibitem{Beffara}
\begin{barticle}[mr]
\bauthor{\bsnm{Beffara},~\bfnm{Vincent}\binits{V.}}
(\byear{2008}).
\btitle{The dimension of the {SLE} curves}.
\bjournal{Ann. Probab.}
\bvolume{36}
\bpages{1421--1452}.
\bid{doi={10.1214/07-AOP364}, issn={0091-1798}, mr={2435854}}
\bptok{imsref}%
\end{barticle}
\endbibitem

\bibitem{Cardynotes}
\begin{barticle}[mr]
\bauthor{\bsnm{Cardy},~\bfnm{John}\binits{J.}}
(\byear{2005}).
\btitle{S{LE} for theoretical physicists}.
\bjournal{Ann. Physics}
\bvolume{318}
\bpages{81--118}.
\bid{doi={10.1016/j.aop.2005.04.001}, issn={0003-4916}, mr={2148644}}
\bptok{imsref}%
\end{barticle}
\endbibitem

\bibitem{DM}
\begin{bbook}[mr]
\bauthor{\bsnm{Dellacherie},~\bfnm{Claude}\binits{C.}} \AND
  \bauthor{\bsnm{Meyer},~\bfnm{Paul-Andr{\'e}}\binits{P.-A.}}
(\byear{1982}).
\btitle{Probabilities and Potential: Theory of Martingales. B}.
\bseries{North-Holland Mathematics Studies}
\bvolume{72}.
\bpublisher{North-Holland}, \baddress{Amsterdam}.
\bid{mr={0745449}}
\bptok{imsref}%
\end{bbook}
\endbibitem

\bibitem{GKnotes}
\begin{barticle}[mr]
\bauthor{\bsnm{Gruzberg},~\bfnm{Ilya~A.}\binits{I.~A.}} \AND
  \bauthor{\bsnm{Kadanoff},~\bfnm{Leo~P.}\binits{L.~P.}}
(\byear{2004}).
\btitle{The {L}oewner equation: Maps and shapes}.
\bjournal{J.~Stat. Phys.}
\bvolume{114}
\bpages{1183--1198}.
\bid{doi={10.1023/B:JOSS.0000013973.40984.3b}, issn={0022-4715}, mr={2039475}}
\bptok{imsref}%
\end{barticle}
\endbibitem

\bibitem{karatzasshreve}
\begin{bbook}[mr]
\bauthor{\bsnm{Karatzas},~\bfnm{Ioannis}\binits{I.}} \AND
  \bauthor{\bsnm{Shreve},~\bfnm{Steven~E.}\binits{S.~E.}}
(\byear{1991}).
\btitle{Brownian Motion and Stochastic Calculus},
\bedition{2nd} ed.
\bseries{Graduate Texts in Mathematics}
\bvolume{113}.
\bpublisher{Springer}, \baddress{New York}.
\bid{doi={10.1007/978-1-4612-0949-2}, mr={1121940}}
\bptok{imsref}%
\end{bbook}
\endbibitem

\bibitem{LWerness}
\begin{barticle}[auto:STB|2012/05/30|10:51:56]
\bauthor{\bsnm{Lawler},~\bfnm{F.}\binits{F.}} \AND
  \bauthor{\bsnm{Werness},~\bfnm{B.}\binits{B.}}
(\byear{2013}).
\btitle{Multi-point Green's functions for SLE and an estimate of
  Beffara}.
\bjournal{Ann. Probab.}
\bvolume{41}
\bpages{1513--1555}.
\bptok{imsref}%
\end{barticle}
\endbibitem

\bibitem{LPark}
\begin{bincollection}[mr]
\bauthor{\bsnm{Lawler},~\bfnm{G.}\binits{G.}}
(\byear{2009}).
\btitle{Schramm--{L}oewner evolution ({SLE})}.
In \bbooktitle{Statistical Mechanics}
(\beditor{S.~Sheffield} and
\beditor{T. Spencer}, eds.).
\bseries{IAS/Park City Mathematics Series}
\bvolume{16}
\bpages{231--295}.
\bpublisher{Amer. Math. Soc.}, \baddress{Providence, RI}.
\bid{mr={2523461}}
\bptok{imsref}%
\end{bincollection}
\endbibitem

\bibitem{lawlerwalk}
\begin{barticle}[mr]
\bauthor{\bsnm{Lawler},~\bfnm{Gregory~F.}\binits{G.~F.}}
(\byear{1980}).
\btitle{A self-avoiding random walk}.
\bjournal{Duke Math. J.}
\bvolume{47}
\bpages{655--693}.
\bid{issn={0012-7094}, mr={0587173}}
\bptok{imsref}%
\end{barticle}
\endbibitem

\bibitem{lawlerbook}
\begin{bbook}[mr]
\bauthor{\bsnm{Lawler},~\bfnm{Gregory~F.}\binits{G.~F.}}
(\byear{2005}).
\btitle{Conformally Invariant Processes in the Plane}.
\bseries{Mathematical Surveys and Monographs}
\bvolume{114}.
\bpublisher{Amer. Math. Soc.}, \baddress{Providence, RI}.
\bid{mr={2129588}}
\bptok{imsref}%
\end{bbook}
\endbibitem

\bibitem{LSWlerw}
\begin{barticle}[mr]
\bauthor{\bsnm{Lawler},~\bfnm{Gregory~F.}\binits{G.~F.}},
  \bauthor{\bsnm{Schramm},~\bfnm{Oded}\binits{O.}} \AND
  \bauthor{\bsnm{Werner},~\bfnm{Wendelin}\binits{W.}}
(\byear{2004}).
\btitle{Conformal invariance of planar loop-erased random walks and uniform
  spanning trees}.
\bjournal{Ann. Probab.}
\bvolume{32}
\bpages{939--995}.
\bid{doi={10.1214/aop/1079021469}, issn={0091-1798}, mr={2044671}}
\bptok{imsref}%
\end{barticle}
\endbibitem

\bibitem{LStime}
\begin{barticle}[mr]
\bauthor{\bsnm{Lawler},~\bfnm{Gregory~F.}\binits{G.~F.}} \AND
  \bauthor{\bsnm{Sheffield},~\bfnm{Scott}\binits{S.}}
(\byear{2011}).
\btitle{A natural parametrization for the {S}chramm--{L}oewner evolution}.
\bjournal{Ann. Probab.}
\bvolume{39}
\bpages{1896--1937}.
\bid{doi={10.1214/10-AOP560}, issn={0091-1798}, mr={2884877}}
\bptok{imsref}%
\end{barticle}
\endbibitem

\bibitem{RS}
\begin{barticle}[mr]
\bauthor{\bsnm{Rohde},~\bfnm{Steffen}\binits{S.}} \AND
  \bauthor{\bsnm{Schramm},~\bfnm{Oded}\binits{O.}}
(\byear{2005}).
\btitle{Basic properties of {SLE}}.
\bjournal{Ann. of Math. (2)}
\bvolume{161}
\bpages{883--924}.
\bid{doi={10.4007/annals.2005.161.883}, issn={0003-486X}, mr={2153402}}
\bptok{imsref}%
\end{barticle}
\endbibitem

\bibitem{Schramm}
\begin{barticle}[mr]
\bauthor{\bsnm{Schramm},~\bfnm{Oded}\binits{O.}}
(\byear{2000}).
\btitle{Scaling limits of loop-erased random walks and uniform spanning trees}.
\bjournal{Israel J. Math.}
\bvolume{118}
\bpages{221--288}.
\bid{doi={10.1007/BF02803524}, issn={0021-2172}, mr={1776084}}
\bptok{imsref}%
\end{barticle}
\endbibitem

\bibitem{SSfree}
\begin{barticle}[mr]
\bauthor{\bsnm{Schramm},~\bfnm{Oded}\binits{O.}} \AND
  \bauthor{\bsnm{Sheffield},~\bfnm{Scott}\binits{S.}}
(\byear{2009}).
\btitle{Contour lines of the two-dimensional discrete {G}aussian free field}.
\bjournal{Acta Math.}
\bvolume{202}
\bpages{21--137}.
\bid{doi={10.1007/s11511-009-0034-y}, issn={0001-5962}, mr={2486487}}
\bptok{imsref}%
\end{barticle}
\endbibitem

\bibitem{schrammzhoudimension}
\begin{barticle}[mr]
\bauthor{\bsnm{Schramm},~\bfnm{Oded}\binits{O.}} \AND
  \bauthor{\bsnm{Zhou},~\bfnm{Wang}\binits{W.}}
(\byear{2010}).
\btitle{Boundary proximity of {SLE}}.
\bjournal{Probab. Theory Related Fields}
\bvolume{146}
\bpages{435--450}.
\bid{doi={10.1007/s00440-008-0195-1}, issn={0178-8051}, mr={2574734}}
\bptok{imsref}%
\end{barticle}
\endbibitem

\bibitem{Sperco}
\begin{barticle}[mr]
\bauthor{\bsnm{Smirnov},~\bfnm{Stanislav}\binits{S.}}
(\byear{2001}).
\btitle{Critical percolation in the plane: Conformal invariance, {C}ardy's
  formula, scaling limits}.
\bjournal{C. R. Acad. Sci. Paris S\'er. I Math.}
\bvolume{333}
\bpages{239--244}.
\bid{doi={10.1016/S0764-4442(01)01991-7}, issn={0764-4442}, mr={1851632}}
\bptok{imsref}%
\end{barticle}
\endbibitem

\bibitem{Sising}
\begin{barticle}[mr]
\bauthor{\bsnm{Smirnov},~\bfnm{Stanislav}\binits{S.}}
(\byear{2010}).
\btitle{Conformal invariance in random cluster models. {I}. {H}olomorphic
  fermions in the {I}sing model}.
\bjournal{Ann. of Math. (2)}
\bvolume{172}
\bpages{1435--1467}.
\bid{doi={10.4007/annals.2010.172.1441}, issn={0003-486X}, mr={2680496}}
\bptok{imsref}%
\end{barticle}
\endbibitem

\bibitem{wernernotes}
\begin{bincollection}[mr]
\bauthor{\bsnm{Werner},~\bfnm{Wendelin}\binits{W.}}
(\byear{2004}).
\btitle{Random planar curves and {S}chramm--{L}oewner evolutions}.
In \bbooktitle{Lectures on Probability Theory and Statistics}.
\bseries{Lecture Notes in Math.}
\bvolume{1840}
\bpages{107--195}.
\bpublisher{Springer}, \baddress{Berlin}.
\bid{doi={10.1007/978-3-540-39982-7_2}, mr={2079672}}
\bptok{imsref}%
\end{bincollection}
\endbibitem

\end{thebibliography}
\end{document}